\documentclass{amsart}
\usepackage{amsmath}
\usepackage{amssymb}
  \usepackage{paralist}
  \usepackage{graphics} 
 \usepackage[colorlinks=true]{hyperref}

\usepackage{mathrsfs}
\usepackage{extarrows}
\usepackage{enumerate}

\newtheorem{theorem}{Theorem}[section]
\newtheorem{corollary}[theorem]{Corollary}

\newtheorem{lemma}[theorem]{Lemma}
\newtheorem{proposition}[theorem]{Proposition}

\theoremstyle{definition}
\newtheorem{definition}[theorem]{Definition}
\newtheorem*{remark}{Remark}

\newcommand{\ZZ}{\mathbb{Z}}

\title[Ergodic properties of $N$-continued fractions]
      {Ergodic properties of $N$-continued fractions}

\author[Peng Sun]{}


 \email{sunpeng@cufe.edu.cn}

\thanks{The author is supported by NSFC No. 11571387.}

\begin{document}

\maketitle\ 

\centerline{\scshape Peng Sun}
\medskip
{\footnotesize
 \centerline{China Economics and Management Academy}
   \centerline{Central University of Finance and Economics}
   \centerline{Beijing 100081, China}
} 

\bigskip

\begin{abstract}
    We discuss some ergodic properties of the generalized Gauss transformation $$T_N(x)=\{\frac
{N}{x}\}.$$
We generalize a series of results for the regular
continued fractions, such as Khinchin's constant  and L\'evy's constant.

\end{abstract}


\section{Introduction}

The Gauss transformation 
$$T(x)=\{\frac1x\},$$
where $\{x\}$ denotes the decimal part of $x$,
has been well studied. It has a strong relation with continued fractions
and has numerous applications to number theory, dynamical systems and
etc. This transformation has a unique absolutely continuous ergodic measure
on $I=[0,1)$:
\begin{equation*}
d\mu(x)=\frac{1}{\ln2}\frac{1}{1+x}dm(x),
\end{equation*}
where $m$ is the Lebesgue measure. The expression of the density
function is rather simple.
However, for most maps that may be viewed as generalizations of Gauss transformation,
we do not have such expressions
and in general we do not expect to have a simple explicit expression, even
for some special cases (for instance, $T(x)=\{\dfrac{1}{x^2}\}$. Discussions
on such transformations can be found in \cite{Choe}).
It seems that there is no general method for finding
the explicit expression of the absolutely continuous ergodic measure.

Recently we found an approach and obtained explicit expressions of the
absolutely continuous ergodic measures for
a family of generalized Gauss transformations.
This leads us to a series of interesting results.


\begin{definition}
Let $N$ be a positive integer. Define the 
generalized Gauss transformation
on $I=[0,1)$ by
\begin{equation*}
T_N(x)=\begin{cases}0,& x=0,\\
\{\cfrac N x\},& x\ne 0,
\end{cases}.
\end{equation*}
\end{definition}
\begin{definition}\label{def:gcf}
For $x\in I$, let 
\begin{equation}\label{eq:gcf}
x=\cfrac{N}{a_1+\cfrac N{a_2+\cfrac N{a_3+\cfrac N\ddots}}}=[a_1,a_2,a_3,\cdots]_N
\end{equation}
be the generalized continued fraction of $x$, where $a_k=\dfrac{N}{T_N^{k-1}(x)}-T_N^k(x)$
is the integer part of $\dfrac{N}{T_N^{k-1}(x)}$
for each $k=1,2,\cdots$. We put $a_k=\infty$ if $T_N^{k-1}(x)=0$.
\end{definition}

\begin{remark}
Note that for all $k$, $a_k\ge N$ and
$T_N^k([a_1,a_2,\cdots]_N)=[a_{k+1},a_{k+2},\cdots]_N$.
There is $K$ such that $a_K=\infty$ if and only if $x$ is rational, in which
case $a_k=\infty$ for all $k\ge K$.
\end{remark}

\begin{theorem}\label{thmain}
For every positive integer $N$, $T_N$ has a unique absolutely continuous ergodic
invariant
measure 
$$d\mu_N(x)=\frac1{\ln (N+1)-\ln N}\cdot\frac 1{N+x}dm(x).$$
\end{theorem}

\begin{remark}
Since the first draft of this work had been posted on arXiv, we learned that
there have already been some works on this topic:
\begin{enumerate}
\item Continued fractions of the form \eqref{eq:gcf} were first introduced
in \cite{BG},
in which they were named $N$-continued fractions,
 to obtain periodic expansions for quadratic irrational numbers.
However, in \cite{BG} the authors do not require that $a_k\ge N$ for all
$k$.

\item In \cite{AW}, some basic and number theoretic properties of $N$-continued
fractions are discussed. It is shown that the $N$-continued fraction expansion
of a number $x\in[0,1)$ is induced by $T_N$ (as in Definition \ref{def:gcf}) if and
only if $a_k\ge N$ for all $k$.

\item In \cite{DK} the measure in Theorem \ref{thmain} is presented without
a proof.  The derivation process is not provided either.
\end{enumerate}
\end{remark}

In this paper, we attempt to give an idea about how we find the measures
$\mu_N$ in Section 2 and present a proof of ergodicity in Section 3.
In Section 4 and Section 5 we discuss some results analogous to the so-called
Khinchin's constant and L\'evy's constant for regular continued fractions.
We consider the following theorem to be the highlight of such results:

\begin{theorem}\label{thmaincoeff}
Assume that $x\in I$ and $x=[a_1,a_2,a_3,\cdots]_N$.
\begin{enumerate}
\item Let
$$K_N(x)=\lim_{n\to\infty}\sqrt[n]{\prod_{k=1}^na_k}$$
be the geometric mean of the coefficients. For Lebesgue almost every
$x$, $K_N(x)=K_N$ is a constant that depends only on $N$. In particular,
we have
$$\lim_{N\to\infty}\frac{K_N}{N}=e.$$
\item Let 
$$[a_1,\cdots, a_n]_N=\frac{A_n(x)}{B_n(x)}\text{ with }A_n(x), B_n(x)\in\ZZ^+\text{
(without reduction).}$$
Then:
\begin{enumerate}[(i)]
\item For almost every $x\in I$,
$$\lim_{n\to\infty}\frac1n\ln B_n(x)=L_N(x)=L_N$$
is a  constant that depends only on $N$. In particular,
\begin{equation}\label{eqLN}
\lim_{N\to\infty}(L_N-\ln N)=1.
\end{equation}
\item For every $x\in I$,
$$\liminf_{n\to\infty}\frac1n\ln B_n(x)\ge\ln\frac{\sqrt{N^2+4N}+N}{2}.$$
\end{enumerate}
\end{enumerate}
\end{theorem}

\section{The Invariant Measure}
In this section we try to give an idea about how we find the measures
$\mu_N$.


We assume that $T_N$ has an absolutely continuous invariant measure
$\mu_N$ with continuous density. Then $F(t)=\mu_N([0,t))$ is a smooth function on $[0,1)$. As $\mu_N$
is invariant, we have for every $t\in[0,1)$,
$$\mu_N((0,t))=\mu_N(T_N^{-1}((0,t)))=\sum_{k=p}^\infty\mu_N((\frac{N}{k+t},\frac{N}{k})),$$
i.e.
\begin{equation}\label{Fsum}
F(t)=\sum_{k=N}^\infty(F(\frac{N}{k})-F(\frac{N}{k+t})).
\end{equation}
We assume that $F$ extends to a smooth function on $\mathbb{R}^+\cup\{0\}$
such that
(\ref{Fsum}) holds for all $t\in\mathbb{R}^+\cup\{0\}$. Replace $t$ by $1+t$ in (\ref{Fsum}) and note that $F(0)=0$,
we have
\begin{equation*}\label{Feq}
F(1+t)=\sum_{k=N}^\infty(F(\frac{N}{k})-F(\frac{N}{k+1+t}))=F(t)+F(\frac{N}{N+t}).
\end{equation*}
Let $F(t)=G(1+\cfrac tN)$ for all $t$. Then
$$G(1+\frac{1+t}{N})-G(1+\frac tN)=G(1+\frac{1}{N+t}).$$
Let $G(t)=H(\ln t)$ for all $t$. Then
\begin{equation}\label{Ffinal}
H(\ln(N+1+t)-\ln N)-H(\ln(N+t)-\ln N)=H(\ln(N+1+t)-\ln(N+t)).
\end{equation}
Then we are aware that (\ref{Ffinal}) holds for all linear
functions $H$ that satisfies
$$H(x)-H(y)=H(x-y)\text{ for all }x,y\in\mathbb{R}.$$
So we guess that $H(x)=cx$ for some constant $c\in\mathbb{R}$.
Then $$G(t)=H(\ln t)=c\ln t$$ and 
$$F(t)=G(1+\frac tN)=c\ln(1+\frac tN)=c(\ln(t+N)-\ln N).$$
As $F(1)=\mu_N([0,1))=1$, we have 
$$c=\frac1{\ln (N+1)-\ln N}.$$
Hence for every $x\in[0,1)$,
$$d\mu_N(x)=F'(x)dm(x)=\frac1{\ln (N+1)-\ln N}\cdot\frac 1{N+x}dm(x).$$
It is straightforward to check that $\mu_N$ is indeed
an invariant measure for $T_N$.

\section{Ergodicity of $\mu_N$}\label{erg}
Existence of an absolutely continuous
invariant measure for $T_N$,
as well as ergodicity of the measure, is actually guaranteed by a theorem
of Bowen \cite{Bowen}.
Here we would like to present a proof of ergodicity of $\mu_N$ following \cite[Section 2]{CRN}
to get some insight into the $N$-continued fractions.
First we note the recursive structure of the convergents:

\begin{lemma}\label{lemmapnum}(cf. \cite{JT})
 Let $x=[a_1,a_2,\cdots]_N$ and $\cfrac{A_n}{B_n}=[a_1.a_2,\cdots, a_n]_N$ be the $n$-th convergent of $x$ (assuming no reduction of fraction
is made). For $a_n\ne\infty$,
$$\begin{cases}
A_n=a_nA_{n-1}+pA_{n-2}\\
B_n=a_nB_{n-1}+pB_{n-2}
\end{cases}\text{ and }A_{n-1}B_n-A_nB_{n-1}=(-N)^n.$$
\end{lemma}
\begin{remark}
As $a_n\ge N$ for all $n$, we can see that $B_n\ge N^n$ be induction.
\end{remark}
 As $\mu_N$ is
absolutely continuous with nonzero density, its ergodicity is equivalent
to the following fact:
\begin{theorem}
Let $E$ be a Borel subset of $[0,1)$ such that $E=T_N^{-1}(E)$ and $m(E)=d<1$. Then $m(E)=0$.
\end{theorem}

\begin{proof}Let $E$ be an invariant subset and $m(E)=d<1$. Let $\chi$ be the characteristic function
on $E$. We choose $\xi\in (0,1)$ and write
$$\xi=[a_1,a_2,\cdots]_N.$$
We fix a positive integer $n$ and denote by $\cfrac{r}{q}$ and $\cfrac{r'}{q'}$
the $(2n-1)$-th and $2n$-th convergents of $\xi$ (without reduction), i.e.
$$\frac{r}{q}=[a_1,a_2,\cdots,a_{2n-1}]_N, \frac{r'}{q'}=[a_1, a_2,\cdots
a_{2n}]_N.$$
Let
\begin{equation*}\label{eqTn}
y=[a_1,a_2,\cdots, a_{2n-1},x+a_{2n}]_N=
\frac{\frac Nx r'+Nr}{\cfrac Nxq'+Nq}=
\frac{rx+r'}{qx+q'}.
\end{equation*}
We have $T_N^{2n}(y)=x$. So 
$\chi(x)=\chi(\dfrac{rx+r'}{qx+q'})$
for every
$x\in[0,1)$.

Let $t_1=\cfrac{r'}{q'}$, $t_2=\cfrac{r+r'}{q+q'}$. Then 
$$t_{2}-t_1=\frac{q'(r+r')-r'(q+q')}{q'(q+q')}=\frac{q'r-r'q}{q'(q+q')}=\frac{N^{2n}}{q'(q+q')}\le\frac{1}{N^{2n}}.$$
We have
\begin{align*}
\frac{m(E\cap[t_1,t_2])}{m([t_1,t_2])}=&\frac{q'(q+q')}{N^{2n}}\int_{t_1}^{t_2}\chi(y)dy\\
=&\frac{q'(q+q')}{N^{2n}}\int_0^1\chi(\frac{rx+r'}{qx+q'})d(\frac{rx+r'}{qx+q'})\\
=&\frac{q'(q+q')}{N^{2n}}\int_0^1\chi(x)\cdot\frac{q'r-r'q}{(qx+q')^2}dx\\
=&q'(q+q')\int_0^1\chi(x)\cdot\frac{dx}{(qx+q')^2}.
\end{align*}

As $m(E)=d<1$ and $\cfrac1{(qx+q')^2}$ is decreasing with $x$, we have
\begin{align*}
\frac{m(E\cap[t_1,t_2])}{m([t_1,t_2])}
=&q'(q+q')\int_0^1\chi(x)\cdot\frac{dx}{(qx+q')^2}\\
\le&q'(q+q')\int_0^d\frac{dx}{(qx+q')^2}\\
=&q'(q+q')(\frac{1}{qq'}-\frac{1}{q(qd+q')})\\
=&1-\frac{q'(1-d)}{dq+q'}\\
\le&1-\frac{1-d}{1+d}\\
=&C<1.
\end{align*}
If $\xi$ runs over $(0,1)$ and $n$ runs over all positive integers, the intervals
$[t_1,t_2]$ form a covering, in the sense of Vitali, of $(0,1)$. In the virtue
of Lebesgue density theorem, $m(E)=0$.
\end{proof}

\section{Frequencies and Means}
The invariant measure $\mu_N$ allows us to compute the frequencies
and means of the coefficients (partial denominators)  for  $N$-continued fractions.

\begin{theorem}\label{thcal}
Let $\phi$ be a measurable function on $I$ that is integrable over $\mu_N$. Then
for Lebesgue almost every $x\in I$,
$$\lim_{n\to\infty}\frac1n\sum_{k=0}^{n-1}\phi(T_N^k(x))=\int_{I}\phi d\mu_N=\frac1{\ln (N+1)-\ln N}\int_0^1\frac {\phi(t)}{N+t}dt.$$
\end{theorem}

\begin{proof} Just apply Birkhoff ergodic theorem to the map $T_N$ and the
ergodic measure $\mu_N$. Note that $\mu_N(E)=0$ if and only if $m(E)=0$.
\end{proof}

\begin{remark} The theorem also holds for nonnegative
functions with infinite integrals.
\end{remark}

We can apply theorem \ref{thcal} to get the following results.
\begin{proposition}\label{thcoeff} For Lebesgue almost every $x\in I$, the coefficients
in $N$-continued fractions of $x=[a_1,a_2,\cdots]_N$
satisfy:
\begin{enumerate}
\item For every positive integer $M\ge N$,
 the frequency (limit distribution) of $M$ in the sequence $\{a_k\}$ is
 $$V_N(M)=\lim_{n\to\infty}\frac{|\{k|a_k=M, 1\le k\le n\}|}n=\frac{2\ln(M+1)-\ln
 M-\ln(M+2)}{\ln (N+1)-\ln N}.$$
 \item The expectation of the coefficients is infinity:
 $$\lim_{n\to\infty}\frac1n\sum_{k=1}^n a_k=\infty.$$
 \item The geometric mean of the
 coefficients is
 $$K_N=K_N(x)=\lim_{n\to\infty}\sqrt[n]{\prod_{k=1}^n a_k}=\prod_{k=N}^\infty(\frac{(k+1)^2}{k(k+2)})^{\frac{\ln
 k}{\ln(N+1)-\ln N}}.$$
\end{enumerate}
\end{proposition}

\begin{proof}
(1) can be obtained by setting $\phi$ as the characteristic function on $[\dfrac{N}{M+1},\dfrac{N}{M})$
in Theorem \ref{thcal}, i.e.
\begin{align*}
V_N(M)=\frac1{\ln (N+1)-\ln N}\int_{\frac N{M+1}}^{\frac pM}\frac {1}{N+t}dt=\frac{2\ln(M+1)-\ln
 M-\ln(M+2)}{\ln (N+1)-\ln N}.
\end{align*}
(2)(3) can be derived from (1) with the frequencies or  by setting $\phi$ in Theorem \ref{thcal} as $\phi([a_1, a_2,\cdots]_{N})=a_1$ and
$\phi([a_1, a_2,\cdots]_N)=\ln a_1$ separately.
\end{proof}

As a corollary, we note that the relative
 frequencies are the same for different values of $N$ as long as
they do not
vanish, i.e., For
any integers $M_1$, $M_2$, for all $N\le\min\{M_1, M_2\}$,
$$\frac{V_N(M_1)}{V_N(M_2)}=\frac{2\ln(M_1+1)-\ln
 M_1-\ln(M_1+2)}{2\ln(M_2+1)-\ln
 M_2-\ln(M_2+2)}=C(M_1, M_2),$$
which is independent of $N$.

The frequencies also implies the following fact:
\begin{corollary}
Let $h:\mathbb{Z}^+\to\mathbb{R}$ be a function on positive integers. Let
$$\phi_N([a_1,a_2,\cdots]_N)=h(a_1).$$
Then $\phi_1$ is integrable over $\mu_1$ if and only if $\phi_N$ is integrable
over $\mu_N$ for all $N$. 
\end{corollary}
This implies that for every function $h:\mathbb{Z}^+\to\mathbb{R}$
, convergence of the mean 
$$\frac{1}{n}\sum_{k=1}^n h(a_k)$$
is the
same for the regular continued fraction and for the generalized
ones.

Proposition \ref{thcoeff}(2)(3) and Theorem \ref{thmaincoeff}(1) are part of
the following general result.

\begin{definition}
For 
$x=[a_1, a_2, \cdots]_N\in I$, we define the following H\"older means of
the coefficients (provided the limits exist):
\begin{align*}
\text{For }r\ne0,\ \ &K_{N,r}(x)=\lim_{n\to\infty}(\frac1n\sum_{j=1}^n a_j^r)^{\frac{1}{r}};
\\
\text{For } r=0, \ \ &K_{N,0}(x)=K_N(x)=\lim_{n\to\infty}\sqrt[n]{\prod_{j=1}^n a_j}=\lim_{r\to
0} K_{N,r}(x).
\end{align*}
\end{definition}

Note that the $N$-continued fraction coincides with the regular one for $N=1$.
The well-known result of Khinchin states that for almost every $x\in I$,
$$K_{1}(x)=K_{1}=2.685452\cdots,$$
where $K_1$ is the so called Khinchin's constant. It is doubtful whether the
author is the first to notice that $K_1$ is quite close to $e$. Meanwhile,
it is not difficult
to find formulae for the means $K_{N,r}$.
However, no one has realized the neat limit behavior of the means. It is fascinating
to find such a brave new world just by changing $1$ to $N$. Our discovery
is as following:

\begin{theorem}\label{thmain}
For Lebesgue almost every $x\in I$ ($x$ should be irrational), we have the following
results:
\begin{enumerate}
\item For $r\ge 1$, $$K_{N,r}(x)=+\infty.$$ 
\item For $r<1$ and $r\ne 0$, $$K_{N,r}:=K_{N,r}(x)
=(\sum_{k=N}^\infty k^r\frac{\ln(1+\frac1k)-\ln(1+\frac1{k+1})}{\ln(1+\frac1N)})^{\frac1r}.$$
$$\lim_{N\to\infty}\frac{K_{N,r}}{N}=(1-r)^{-\frac1r}.$$
\item 
$$K_{N}:=K_{N}(x)=\prod_{k=N}^\infty(\frac{(k+1)^2}{k(k+2)})^{\frac{\ln
 k}{\ln(N+1)-\ln N}}.$$
$$\lim_{N\to\infty}\frac{K_{N}}{N}=e.$$
\end{enumerate}
\end{theorem}

\begin{proof}

It is easy to see that there is $\Gamma\subset I$ such that
$m(\Gamma)=1$ and Theorem \ref{thcal} holds for every $x\in\Gamma$.
\begin{enumerate}

\item Assume that $r\ge 1$ and $M$ is a positive integer no less than $N$, let
$$g_{M}(x)=\begin{cases}\lfloor\dfrac Nx\rfloor, & x\ge\dfrac NM;\\
0, & \text{otherwise.} \end{cases}$$
Then $(g_M)^r\in L^1(\mu_N)$ as
\begin{align*}\int (g_M)^rd\mu_N=&\sum_{k=N}^M k^r\mu_N((\frac N{k+1},\frac Nk])\\
=&\sum_{k=N}^M k^r\frac{\ln(1+\frac1k)-\ln(1+\frac1{k+1})}{\ln(1+\frac1N)}<\infty.
\end{align*}
Apply Theorem \ref{thcal} for $\phi=(g_M)^{r}$. For every $x\in\Gamma$, we have
\begin{align*}
(K_{N,r}(x))^r=&\lim_{n\to\infty}\frac1n\sum_{j=1}^n a_j^r
\ge\lim_{n\to\infty}\frac1n\sum_{j=1}^n a_j^r\chi_M(a_j)
\\=&\lim_{n\to\infty}\frac1n\sum_{j=0}^{n-1} (g_M(T_N^j))^r=\int (g_M)^r d\mu_N
\end{align*}
for all $M$, where $\chi_M(t)=\begin{cases}1, & t\le M;\\ 0, & t>M.\end{cases}$
But for $r\ge 1$, we have
\begin{align*}
(K_{N,r}(x))^r\ge&\int (g_M)^r d\mu_N\ge\int g_M d\mu_N\\=&\sum_{k=N}^M k\frac{\ln(1+\frac1k)-\ln(1+\frac1{k+1})}{\ln(1+\frac1N)}
\\=&N-M\frac{\ln(1+\frac 1{M+1})}{\ln(1+\frac1N)}+\sum_{k=N+1}^M \frac{\ln(1+\frac1k)}{\ln(1+\frac1N)}
\\=&N-M\frac{\ln(1+\frac 1{M+1})}{\ln(1+\frac1N)}+\frac{\ln(M+1)-\ln(N+1)}{\ln(1+\frac1N)}
\\ \to&\infty \text{ as }M\to\infty.
\end{align*}
So $K_{N,r}(x)=+\infty$ for $r\ge 1$ and for every $x\in\Gamma$.
\item Assume that $r<1$ and $r\ne 0$. Let $g(x)=\lfloor \dfrac Nx\rfloor$ for $x\in (0,1)$. 
Note that $\ln(1+t)\le t$ for all $t\ge 0$. For $r<1$, $(g)^r\in L^1(\mu_N)$ as
\begin{align*}
\int(g)^rd\mu_N=&\sum_{k=N}^\infty k^r\frac{\ln(1+\frac1k)-\ln(1+\frac1{k+1})}{\ln(1+\frac1N)}
\\=&\frac{1}{\ln(1+\frac1N)}\sum_{k=N}^\infty k^r\ln(1+\frac1{k(k+2)})
\\\le&\frac{1}{\ln(1+\frac1N)}\sum_{k=N}^\infty \frac{k^r}{k^2}<\infty.
\end{align*}
Apply Theorem \ref{thcal} for $\phi=(g)^{r}$. For every $x\in\Gamma$, we have
\begin{align*}
K_{N,r}(x)=&(\lim_{n\to\infty}\frac1n\sum_{j=1}^n a_j^r)^{\frac1r}
=(\lim_{n\to\infty}\frac1n\sum_{j=0}^{n-1} (g(T_N^j))^r)^{\frac1r}
\\=&(\int (g)^r d\mu_N)^{\frac1r}=K_{N,r}.
\end{align*}

Now we assume that $r<0$. Proof for the case $0<r<1$ is analogous. For $r<0$,
we have
\begin{align*}
\frac{K_{N,r}}{N}=&\frac1N(\int_0^1(\lfloor\frac Nx\rfloor)^rd\mu_N)^{\frac1r}
\\\ge&\frac1N(\int_0^1(\frac Nx)^r\frac{1}{(N+x)\ln(1+\frac1N)}dx)^{\frac1r}
\\\ge&(\frac{1}{(N+1)\ln(1+\frac1N)}\int_0^1 x^{-r}dx)^{\frac1r}
\\\to& (1-r)^{-\frac1r}\text{ as }N\to\infty,
\end{align*}
and
\begin{align*}
\frac{K_{N,r}}{N}
\le&\frac1N(\int_0^1(\frac Nx-1)^r\frac{1}{(N+x)\ln(1+\frac1N)}dx)^{\frac1r}
\\\le&\frac{N-1}{N}(\frac{1}{N\ln(1+\frac1N)}\int_0^1 x^{-r}dx)^{\frac1r}
\\\to& (1-r)^{-\frac1r}\text{ as }N\to\infty.
\end{align*}
Hence
$$\lim_{N\to\infty}\frac{K_{N,r}}{N}=(1-r)^{-\frac1r}.$$
\item Let $h(x)=\ln\lfloor\dfrac{N}{x}\rfloor$ for $x\in(0,1)$. 
Then $h\in L^1(\mu_N)$ as
\begin{align*}
\int h d\mu_N=&\sum_{k=N}^\infty \ln k\frac{\ln(1+\frac1k)-\ln(1+\frac1{k+1})}{\ln(1+\frac1N)}
\\\le&\frac{1}{\ln(1+\frac1N)}\sum_{k=N}^\infty \frac{\ln k}{k^2}<\infty.
\end{align*}
Apply Theorem \ref{thcal} for $\phi=h$. For every $x\in\Gamma$, we have
\begin{align*}
\ln K_{N}(x)=&\lim_{n\to\infty}\frac1n\sum_{j=1}^n \ln a_j
=\lim_{n\to\infty}\frac1n\sum_{j=0}^{n-1} h(T_N^j)
\\=&(\int h d\mu_N)^{\frac1r}=\ln K_{N}.
\end{align*}

Moreover, we have
$$\ln K_{N}-\ln N=\int(\ln \lfloor \frac Nx\rfloor-\ln N) d\mu_N=\int_0^1
\frac{\ln \lfloor \dfrac Nx\rfloor-\ln N}{(N+x)\ln(1+\frac1N)}dx.$$
Hence
\begin{align*}
\ln K_{N}-\ln N
\le&
\int_0^1\frac{\ln\dfrac Nx-\ln N}{N\ln(1+\frac1N)}dx
\\=&\frac{1}{N\ln(1+\frac1N)}\int_0^1(-\ln x)dx
\\\to&1\text{ as }N\to\infty,
\end{align*}
and
\begin{align*}
\ln K_{N}-\ln N\ge
&\int_0^1\frac{\ln(\dfrac Nx-1)-\ln N}{(N+x)\ln(1+\frac1N)}dx
\\=&\int_0^1\frac{-\ln x}{(N+x)\ln(1+\frac1N)}dx-
\int_0^1\frac{\ln N-\ln(N-x)}{(N+x)\ln(1+\frac1N)}dx
\\\ge&\int_0^1\frac{-\ln x}{(N+1)\ln(1+\frac1N)}dx-
\int_0^1\frac{\ln N-\ln(N-1)}{N\ln(1+\frac1N)}dx
\\=&\frac{1}{(N+1)\ln(1+\frac1N)}-\frac{\ln N-\ln(N-1)}{N\ln(1+\frac1N)}
\\\to& 1\text{ as }N\to\infty.
\end{align*}
This implies that
$$\lim_{N\to\infty}(\ln K_{N}-\ln N)=1.$$
Equivalently,
$$\lim_{N\to\infty}\frac{K_{N}}{N}=e.$$
\end{enumerate}
\end{proof}

According to Theorem \ref{thmain}, as $p\to\infty$, besides the geometric mean $$K_{N}\sim eN,$$ we have:
\begin{align*}
\text{(harmonic mean)  }& K_{N,-1}\sim 2N,\\
\text{($\frac12$-H\"older mean)  }& K_{N,\frac12}\sim4N
\end{align*}
and so on.

Some computation shows that
\begin{align*}
K_{1}=&2.685452\cdots=1(2.685452\cdots),\\
K_{2}=&5.412652\cdots=2(2.706326\cdots),\\
K_{3}=&8.136460\cdots=3(2.712153\cdots).
\end{align*}
We can see that the convergence of $\dfrac{K_{N}}{N}$ to $e$ is quite fast.

\section{Lypunov Exponents and Growth of Denominators}
In this section we discuss the Lypunov exponents for the map
$T_N$ and the growth of the denominators in the convergents.
\begin{theorem}
For 
Lebesgue almost every $x\in I$, the Lyapunov exponent
of $T_N$
along the orbit of $x$  is
$$\lambda(T_N)=\lim_{n\to\infty}\frac1n\sum_{k=0}^{n-1}\ln|T_{N}'(T_N^k(x))|=2\Lambda_N+
\ln N,$$
where
$$\Lambda_N=\Lambda(\frac1N), \Lambda(x)=\dfrac{\Theta(x)}{\ln(1+x)}$$
and
$$\Theta(x)=\int_{0}^x\frac{\ln(t+1)}tdt=\sum_{k=1}^\infty\frac{(-1)^{k-1}}{k^2}x^k$$
is a dilogarithm function.
\end{theorem}
\begin{proof}Apply Theorem \ref{thcal} for $\phi=\ln|T_N'|$:
\begin{align*}
\lambda(T_N)
=&\int_{[0,1)}\ln|T_N'(x)| d\mu_N\\
=&\int_0^1(\ln\frac{N}{x^2})\frac1{\ln (N+1)-\ln N}\cdot\frac 1{x+N}dx\\
=&\frac1{\ln (N+1)-\ln N}\int_{0}^{\frac1N}\frac{-2\ln t-\ln N}{t+1}dt\\
=&\frac1{\ln (N+1)-\ln N}(\int_{0}^{\frac1N}\frac{2\ln(t+1)}tdt-(2\ln t\ln(t+1)+\ln
N\ln (t+1))|_0^{\frac1N})\\
=&\frac2{\ln (N+1)-\ln N}\int_{0}^{\frac1N}\frac{\ln(t+1)}tdt+\ln N\\
\\=&2\Lambda_N+\ln N.
\end{align*}
\end{proof}

We know that $\Theta(1)=\cfrac{\pi^2}{12}$ and hence $\Lambda_1=\cfrac{\pi^2}{12\ln2}$
is the logarithm of L\'evy's constant. As $\Theta$ is increasing, we have
$\Lambda_N<\infty$ for all $N$. L\'evy's constant is generalized as
following:
\begin{theorem}\label{thlevy}
Let $x=[a_1,a_2,\cdots]_N$ and $\dfrac{A_n}{B_n}=[a_1,a_2,\cdots,a_n]_N$
(without reduction).
For Lebesgue almost every $x\in I$,
$$\lim_{n\to\infty}\frac{\ln B_n}{n}=\Lambda_N+\ln N.$$
\end{theorem}

\begin{proof}
From Lemma \ref{lemmapnum} we have for every irrational $x\in I$ and every $n$,
$$x=\frac{\frac{N}{T_N^n(x)}A_n+NA_{n-1}}{\frac{N}{T_N^n(x)}B_n+NB_{n-1}}=\frac{A_{n-1}T_N^n(x)+A_n}{B_{n-1}T_N^n(x)+B_n}.$$
Make a substitution and we have
$$T_N^n(x)=-\frac{B_{n}x-A_{n}}{B_{n-1}x-A_{n-1}}.$$
But for Lebesgue almost every $x\in[0,1)$,
\begin{align*}
\lambda(T_N)=&\lim_{n\to\infty}\frac1n\sum_{k=1}^{n}\ln|T_N'(T_N^k(x))|
\\=&\lim_{n\to\infty}\frac1n\sum_{k=1}^{n}\ln|\frac{N}{(T_N^k(x))^2}|
\\=&\ln N-\lim_{n\to\infty}\frac2n\sum_{k=1}^{n}\ln|T_N^k(x)|
\\=&\ln N-\lim_{n\to\infty}\frac2n\sum_{k=1}^{n}\ln|\frac{B_{n}x-A_{n}}{B_{n-1}x-A_{n-1}}|
\\=&\ln N-\lim_{n\to\infty}\frac2n\ln|B_nx-A_n|.
\end{align*}

We have
\begin{align*}
B_nx-A_n=&B_n\frac{A_{n}+T_N^{n-1}(x)A_{n-1}}{B_n+T_N^{n-1}(x)B_{n-1}}-A_n
=\frac{(-N)^nT_N^{n-1}(x)}{B_n+T_N^{n-1}(x)B_{n-1}}.
\end{align*}
Then
$$|B_nx-A_n|\le\frac{N^n}{B_n}, $$
$$\liminf_{n\to\infty}\frac1n\ln|B_nx-A_n|\le\liminf_{n\to\infty}\frac1n\ln\frac{N^n}{B_n}$$
and
$$|B_nx-A_n|\ge\frac{N^n\frac{N}{a_n+1}}{2B_n}>\frac{N^{n+1}}{4B_{n+1}},$$
$$\limsup_{n\to\infty}\frac1n\ln|B_nx-A_n|\ge\limsup_{n\to\infty}\frac1n\ln\frac{N^{n+1}}{4B_{n+1}}=\limsup_{n\to\infty}\frac1n\ln\frac{N^n}{B_n}.$$
Hence
\begin{equation}\label{precis}
\lim_{n\to\infty}\frac1n\ln|B_nx-A_n|
=\lim_{n\to\infty}\frac1n\ln\frac{N^n}{B_n}=\ln N-\lim_{n\to\infty}\frac{\ln
B_n}{n}
\end{equation}
and
$$\lambda(T_N)=2\Lambda_N+\ln N=2\lim_{n\to\infty}\frac{\ln B_n}{n}-\ln N.$$
The result follows.
\end{proof}
\begin{remark}
By (\ref{precis}), for Lebesgue almost every $x\in[0,1)$, 
\begin{equation}\label{rateconv}
\lim_{n\to\infty}-\frac1n\ln|x-\frac{A_n}{B_n}|
=2\lim_{n\to\infty}\frac{\ln
B_n}{n}-\ln N=2\Lambda_N+\ln N=\lambda(T_N).
\end{equation}
This measures the precision of the $n$-th convergents in the $N$-continued fractions, which is a generalization
of Loch's constant (the generalized constant is $\dfrac{\ln 10}{2\Lambda_N+
\ln N}$). As $\lambda(T_N)$ increases with $N$ (see the remark
after Proposition \ref{propinc}), the larger $N$
is, the faster the convergents converge.

(\ref{rateconv}) may also be obtained directly by Shannon-McMillan-Breiman
Theorem. Our approach finds its advantage in the equations \eqref{genest1}.
\end{remark}

We note the following facts about $\Lambda_N$:
\begin{proposition}
$\lim_{N\to\infty}\Lambda_N=1$.
\end{proposition}
\begin{remark}
This verifies \eqref{eqLN} in Theorem \ref{thmaincoeff}.
\end{remark}
\begin{proof} Note that $\Theta'(x)=\cfrac{\ln(1+x)}{x}$. So
\begin{align*}
\lim_{N\to\infty}\Lambda_N=\lim_{x\to0}\frac{\Theta(x)}{\ln(1+x)}
=\lim_{x\to 0}\frac{\Theta'(x)}{\frac{1}{1+x}}
=\lim_{x\to 0}\frac{\ln(1+x)}{x}
=1.
\end{align*}
\end{proof}
\begin{remark}
More computation shows that as $N\to\infty$,
$$\Lambda_N=1+\frac1{4N}-\frac{7}{72N^2}+\frac{1}{18N^3}+O(\frac1{N^4}),$$
which provides a good estimate for $\Lambda_N$ when $N$ is large.
\end{remark}
\begin{proposition}\label{propinc}
$\Lambda_N$ decreases as $N$ increases.
\end{proposition}
\begin{proof}
It suffices to show that $\Lambda'(x)>0$ for all $x\in(0,1)$. We have
\begin{align*}
&\Lambda'(x)>0 \text{ for
all }x\in (0,1)
\\ \iff&\frac1x>\frac{\Theta(x)}{(1+x)\ln^2(1+x)}\text{ for
all }x\in (0,1)
\\ \iff&\frac{(1+x)\ln^2(1+x)}{x}>\Theta(x)\text{ for
all }x\in (0,1)
\\ \Longleftarrow &\frac{x(\ln^2(1+x)+2\ln(1+x))-(1+x)\ln^2(1+x)}{x^2}>\Theta'(x)
\text{ for
all }x\in (0,1)
\\ \iff&2x-\ln(1+x)>x\text{ for
all }x\in (0,1)
\\ \iff&x>\ln(1+x)\text{ for
all }x\in (0,1)
\\ \Longleftarrow &1>\frac{1}{1+x}\text{ for
all }x\in (0,1).
\end{align*}
The last inequality is true.
\end{proof}

\begin{remark}
However, a similar argument shows that
$\lambda(T_N)=2\Lambda_N+\ln N$ increases with $N$.
\end{remark}

\section{Lower Bounds for Non-regular Points}





Ergodic theory does not tell us everything. There are orbits along which
the Lyapunov exponents
do not exist (for rational numbers, or irrationals for which the limit
do not exist) or they differ from $\lambda(T_N)$. Analogous to \cite{Cor},
we have some estimates along such non-regular orbits.

Denote by $\lambda_N(x)$ the Lyapunov exponent of $T_N$ along the orbit of
$x$. Let
$$\lambda_N^-(x):=\liminf_{n\to\infty}\frac1n\sum_{k=0}^{n-1}\ln|T_{N}'(T_N^k(x))|,$$
which coincides with $\lambda_N(x)$ if the Lyapunov exponent exists.
For example, for
the fixed point $$z_{N,p}=[p,p,p,\cdots]_N=\frac{\sqrt{p^2+4N}-p}{2},$$ where $p\ge N$, we have
$$\lambda_N(z_{N,p})=\lambda_N^-(z_{N,p})=\ln|T_N'(z_{N,p})|
=2\ln\frac{\sqrt{p^2+4N}+p}{2}-\ln
N.$$
So there are orbits with arbitrarily large Lyapunov exponents.
Then we have a generalization of a theorem in \cite{Cor} that gives a
lower bound for $\lambda_N^-$.

\begin{theorem}
For every $x\in I$,
$$\lambda_N^-(x)\ge 2\ln\frac{\sqrt{N+4}+\sqrt{N}}{2}.$$
\end{theorem}
\begin{proof}
Let $x=[a_1, a_2,\cdots]_N$ such that $\lambda_N(x)$ exists and denote $x_k=T_N^k(x)$.
Then
$$\lambda_N^-(x)=\liminf_{n\to\infty}\frac1n\sum_{k=0}^{n-1}\ln|T_N'(x_k)|=\liminf_{n\to\infty}\frac1n\sum_{k=0}^{n-1}\ln\frac{N}{x_k^2}.$$
But for each $k$, we have either
$$\frac{N}{x_k^2}\ge(\frac{\sqrt{N+4}+\sqrt{N}}{2})^2\iff x_k\le\frac{\sqrt{N^2+4N}-N}{2},$$
or $$x_k>\frac{\sqrt{N^2+4N}-N}{2}.$$
In the latter case, we have $x_{k}=\cfrac{N}{a_{k+1}+x_{k+1}}$ and hence
$$\frac{x_{k+1}}{a_{k+1}+x_{k+1}}=1-\frac{a_{k+1}}{N}x_k<1-\frac{\sqrt{N^2+4N}-N}{2}=(\frac{\sqrt{N+4}-\sqrt{N}}{2})^2,$$
$$\frac{N}{x_k^2}\frac{N}{x_{k+1}^2}=(\frac{a_{k+1}+x_{k+1}}{x_{k+1}})^2>(\frac{\sqrt{N+4}+\sqrt{N}}{2})^4.$$
Note that $x_k<1$ for all $k$. Therefore for every $n$, we always have
$$\prod_{k=0}^{n-1}\frac{N}{x_k^2}\ge((\frac{\sqrt{N+4}+\sqrt{N}}{2})^{2})^{n-1}.$$
Hence $$\lambda_N^-(x)\ge 2\ln\frac{\sqrt{N+4}+\sqrt{N}}{2}.$$
\end{proof}
\begin{remark}
The lower bound is achieved at infinitely many points, for example, the preimages
of $$z_{N,N}=[N,N,N,\cdots]_N=\frac{\sqrt{N^2+4N}-N}{2},$$
which are analogs of the so-called Noble numbers (for the regular
continued fraction). 
\end{remark}

From the proof of Theorem \ref{thlevy}, we can see that the following estimate holds
for every $x\in I$:


\begin{equation}\label{genest1}
\liminf_{n\to\infty}\frac{\ln B_n}{n}\ge\frac12(\ln N-\limsup_{n\to\infty}\frac1n|B_nx-A_n|)
\ge
\frac{\lambda_N^-(x)+
\ln N}{2}
\end{equation}
and if $\lambda_N(x)$ exists,
\begin{equation*}
\lim_{n\to\infty}\frac{\ln B_n}{n}=
\frac{\lambda_N(x)+
\ln N}{2}.
\end{equation*}
So we have:

\begin{corollary}
For every $x\in I$,
$$\liminf_{n\to\infty}\frac{\ln B_n}{n}\ge\ln\frac{\sqrt{N^2+4N}+N}{2}.$$
\end{corollary}

\end{document}